\documentclass[twoside]{irmaems}
\usepackage{amssymb} 
\usepackage{amsmath} 
\usepackage{latexsym}
\usepackage{graphicx} 

\renewcommand\theequation{\arabic{section}.\arabic{equation}}

\theoremstyle{definition} 

 \newtheorem{definition}{Definition}[section]

\theoremstyle{plain}      

 \newtheorem{proposition}[definition]{Proposition}
 \newtheorem{theorem}[definition]{Theorem}
 
 \newtheorem{lemma}[definition]{Lemma}

\def\sb{\mathbb{S}}
\def\tb{\mathbb{T}}

\begin{document}

\title{Beta-conjugates of real algebraic numbers 
as Puiseux expansions}

\author{
Jean-Louis Verger-Gaugry~
}

\address{
}

\maketitle

\begin{abstract}
The beta-conjugates of a base of numeration $\beta > 1$, $\beta$ being 
a Parry number,
were introduced by Boyd, in the context of the R\'enyi-Parry dynamics of
numeration system and the beta-transformation. 
These beta-conjugates are canonically associated with
$\beta$.
Let $\beta > 1$ be a real algebraic number. A
more general definition of the beta-conjugates of 
$\beta$ is introduced in terms of
the Parry Upper function $f_{\beta}(z)$ of the beta-transformation.
We introduce 
the concept of a germ of curve at 
$(0,1/\beta) \in \mathbb{C}^{2}$
associated 
with $f_{\beta}(z)$
and
the reciprocal of the
minimal polynomial of
$\beta$.
This germ 
is decomposed into irreducible elements 
according to the theory of Puiseux, gathered into 
conjugacy classes.
The beta-conjugates of $\beta$, in terms of the
Puiseux expansions, are given a new equivalent 
definition in this new context.
If $\beta$ is a Parry number the (Artin-Mazur) dynamical 
zeta function
$\zeta_{\beta}(z)$ of the beta-transformation, simply related
to $f_{\beta}(z)$, 
is  
expressed as a product formula, under some assumptions,
a sort of analog to the 
Euler product of the Riemann zeta function,
and the factorization of the
Parry polynomial of $\beta$ is deduced 
from the germ.
\end{abstract}



\newpage
\tableofcontents  

\section{Introduction}
\label{S1}

For $\beta > 1$ a Parry number, Boyd \cite{boyd2} introduced
the notion of the beta-conjugates of $\beta$ 
in the context of the R\'enyi - Parry
numeration system
\cite{renyi} \cite{parry} \cite{blanchard} \cite{frougny1}.
As he has shown it in numerous examples, the
investigation of beta-conjugates 
is an important question. 
These beta-conjugates, 
up till now defined for Parry numbers,
are canonically associated to $\beta$ and to the dynamics of
the beta-transformation.
Our aim is to show that their definition 
can be given in a larger context, 
namely for any algebraic number $\beta > 1$,
and that
the theory of Puiseux provides a geometric origin
to the beta-conjugates
of $\beta$; for doing it, once
$\beta$ is given by its minimal polynomial, we
first put into evidence that 
a germ of curve ``at $1/\beta$" does exist and
develop new tools deduced from the canonical decomposition of this germ 
in order to express the beta-conjugates of $\beta$
in terms of the Puiseux expansions
\cite{puiseux} \cite{casasalvero} of the germ.

Though the existence of this germ of curve 
was discovered by the author some years ago, the present note 
is the first account on it 
and its potential applications. It establishes a deep relation between
the theory of singularities of curves in Algebraic Geometry and
the dynamical system of numeration $([0,1], T_{\beta})$ where
$\beta > 1$ is an algebraic number and 
$T_{\beta}$ is the beta-transformation. 
The existence of this germ of curve 
brings new tools to the 
R\'enyi-Parry numeration system, namely the Puiseux series
associated to the germ, and 
defines new directions of research
for old questions. For instance, if $(\beta_i)$ is a sequence
of Salem numbers which converges to a real number $\beta$, then
it is known \cite{bertinetal} that 
$\beta$ is a Pisot or a Salem number,
but how is distributed the collection of the
beta-conjugates and the Galois conjugates of
$\beta_i$, with $i$ large enough,
with respect to that of the limit $\beta$ ? 
This question is merely a generalization of
the classical question of how is distributed the 
collection of the Galois conjugates
of $\beta_i$ with respect to that of $\beta$ ?
Why should we add the beta-conjugates ?
Because a new phenomenon appears which generally does not exist
with only the Galois conjugates:
under some assumptions 
the collections of Galois- and beta- conjugates
may have equidistribution limit properties on the unit circle
(\S 3.6 in \cite{vergergaugry2}) if
the two collections of
conjugates are simultaneously considered. 
Both collections of conjugates
are expected to play a role in limit and dynamical properties of
convergent sequences of real algebraic numbers $> 1$
in general.
A basic question is then to
understand the role and the relative density of the
beta-conjugates 
in this possible equidistribution process, in particular if
the limit $\beta$ is an integer $\geq 2$ or is
equal to $1$ (context of the Conjecture of Lehmer). 

Conversely the curve 
canonically associated with this numeration dynamical system 
is of interest for itself (critical points, monodromy, ...). 
It will be studied elsewhere.

In this first contribution we
obtain useful expressions for the beta-conjugates as Puiseux 
expansions of $\beta$ and of 
the minimal polynomial of $\beta$, towards this goal.

As usual now we use the new terminology, which is 
in honor of W. Parry.
The old terminology used by W. Parry himself in \cite{parry}
transforms as follows:
we now call {\it Parry number} a $\beta$-number \cite{parry}, and 
{\it Parry polynomial} of a Parry number $\beta$ the 
characteristic polynomial \cite{parry} of the $\beta$-number
$\beta$. As previously a {\it simple} Parry number $\beta$ is 
a Parry number $\beta$ for which the R\'enyi $\beta$-expansion 
$d_{\beta}(1)$ of unity is finite 
(i.e. ends in infinitely many zeros). The exact definitions 
are given in Section \ref{S3}.

If $\beta$ is a Parry number, 
the roots of the Parry polynomial of $\beta$, denoted by
$\beta^{(i)}$, 
are called
the conjugates of $\beta$. A conjugate of $\beta$
is either a Galois conjugate of $\beta$ or a beta-conjugate, 
if the collection
of beta-conjugates of $\beta$ is not empty.

Let $\beta > 1$ be a real number and $d_{\beta}(1)= 0. t_1 t_2 t_3 \ldots$
be the R\'enyi $\beta$-expansion of 1.  Since this 
R\'enyi $\beta$-expansion
of 1 controls the language in base
$\beta$ \cite{lothaire}, the properties of the analytic function
constructed from it, called Parry Upper function 
at $\beta$, defined by $f_{\beta}(z) := -1 + \sum_{i \geq 1} t_i z^i$, 
is of particular importance.

Ito and Takahashi \cite{itotakahashi} have shown that the
Parry Upper function at a Parry number $\beta$, 
of the complex variable $z$, is related
to the (Artin-Mazur) dynamical zeta function 
\begin{equation}
\label{zeta00}
\zeta_{\beta}(z) ~:=~ \exp \left(
\sum_{i \geq 1} \frac{\#\{x \in [0, 1] 
\mid T_{\beta}^n(x) = x\}}{n} z^n\right)
\end{equation}
of the beta-transformation $T_{\beta}$
(Artin and Mazur \cite{artinmazur}, Boyd \cite{boyd2}, Flatto, Lagarias and Poonen
\cite{flattolagariaspoonen},
Verger-Gaugry \cite{vergergaugry1}
\cite{vergergaugry2}). 
Namely, if $\beta$
is a nonsimple Parry number, with 
$d_{\beta}(1) = 0 . t_1 t_2 \ldots t_m (t_{m+1} \ldots t_{m+p+1})^{\omega}$
(where $(\,)^{\omega}$ means infinitely repeated),
\begin{equation}
\label{fzetaNONsimple}
f_{\beta}(z) ~=~ - \frac{1}{\zeta_{\beta}(z)} ~=~ 
- \frac{P_{\beta,P}^{*}(z)}{1-z^{p+1}}
\end{equation}
where $P_{\beta,P}^{*}(X) 
= (-1)^{d_P} \left(\prod_{i=1}^{d_P} \beta^{(i)}\right) 
\, \times \, \prod_{i=1}^{d_P} (X- \frac{1}{\beta^{(i)}})
= X^{d_P} P_{\beta,P}(1/X)$
is the reciprocal of the Parry polynomial 
$P_{\beta,P}(X)$ of $\beta$,
of degree $d_P = m + p + 1$ ($m$ is the preperiod length and
$p+1$ is the period length in $d_{\beta}(1)$, if $\beta$ is a nonsimple Parry number,
with the convention
$p+1=0$ for a finite R\'enyi $\beta$-expansion of unity 
(for $\beta$ a simple Parry number), with 
the convention $m=0$ if 
$d_{\beta}(1)$ is a purely periodic expansion
\cite{vergergaugry2});
if $\beta$ is a simple Parry number, with
$d_{\beta}(1) = 0 . t_1 t_2 \ldots t_m$, then
\begin{equation}
\label{fzetaSIMPLE}
f_{\beta}(z) ~=~ - \frac{1-z^m}{\zeta_{\beta}(z)} ~=~ 
- P_{\beta,P}^{*}(z).
\end{equation}
The zeros of $f_{\beta}(z)$ are 
the poles of $\zeta_{\beta}(z)$. 
The set of zeros of 
$f_{\beta}(z)$ is the set $(1/\beta^{(i)})_i$
of the reciprocals of the
conjugates $(\beta^{(i)})_i$
of $\beta$. The geometry of the conjugates 
$(\beta^{(i)})_i$ of $\beta$ was carefully studied
by Solomyak \cite{solomyak} \cite{vergergaugry1}: 
these conjugates all lie in
Solomyak's fractal $\Omega$, a compact connected subset of the closed disc
$\overline{D}(0, \frac{1+\sqrt{5}}{2})$ in the complex plane
(Figure \ref{fractalSOLO}), 
having a cusp at $z=1$, a spike
on the negative real axis, 
symmetrical with respect to the real line \cite{solomyak}
\cite{vergergaugry2}.

If $\beta > 1$ is an algebraic number but 
not a Parry number, some relations are
expected between $f_{\beta}(z)$ and $\zeta_{\beta}(z)$, 
though not yet determined. 
Indeed, on one hand, $f_{\beta}(z)$ is an
analytic function on the open unit disc
which admits
$|z|=1$ as natural boundary by 
Szeg\H{o}-Carlson-Poly\'a's Theorem \cite{dienes} \cite{vergergaugry2};
$f_{\beta}(z)$ 
admits $1/\beta$ as zero of multiplicity one, 
which is its only zero in the interval
$(0,1)$. On the other hand $\zeta_{\beta}(z)$ 
is an analytic function defined on
the open unit disc $D(0,1/\beta)$, 
which admits a nonzero meromorphic continuation 
on $D(0,1)$, by \cite{haydn} \cite{parrypollicott} \cite{ruelle},
or by Baladi-Keller's Theorem 2 in
\cite{baladikeller}.
Whether the zeros of $f_{\beta}(z)$ correspond to poles of
$\zeta_{\beta}(z)$ is unknown.
The behaviour of the dynamical zeta function
$\zeta_{\beta}(z)$ on the unit circle remains unknown, 
i.e. we do not know  
whether 
$|z|=1$ is a natural boundary for
$\zeta_{\beta}(z)$ or not. But the multiplicity
of the pole $1/\beta$ of $\zeta_{\beta}(z)$ 
is known to be one 
\cite{haydn} \cite{parrypollicott} \cite{ruelle}. 
For $\beta > 1$ an algebraic number,
as a consequence of Theorem 1 in \cite{baladikeller}, the coefficients
in \eqref{zeta00} obey the following asymptotics of growth
(Pollicott, \S 5.2 in \cite{pollicott}) :
for any $\delta > 0$ there exist an integer $M > 0$ and 
constants $(i) ~\lambda_{1,\beta}, \lambda_{2,\beta}, \ldots, \lambda_{M,\beta}$,
with $~|\lambda_{i,\beta}| > 1 + \delta ~(i = 1, \ldots, M)$, and
$(ii)~ C_{1,\beta}, C_{2,\beta}, \ldots, C_{M,\beta} ~\in \mathbb{C}$, such that
\begin{equation}
\label{nombrePointsFixes}
\#\{x \in [0, 1] \mid T_{\beta}^n(x) = x\} ~=~ 
\sum_{i=1}^{M} C_{i,\beta} \lambda_{i,\beta}^{n} + O((1+\delta)^n).
\end{equation}
In the case where $\beta > 1$ is a Parry number, 
$\zeta_{\beta}(z)$ is a rational fraction
and, from \eqref{fzetaNONsimple} and \eqref{fzetaSIMPLE}, 
\eqref{nombrePointsFixes}
transforms into the following exact formula 
(after Pollicott, \S 1 in \cite{pollicott}):
\begin{equation}
\label{nombrePointsFixesRATIONAL}
\#\{x \in [0, 1] \mid T_{\beta}^n(x) = x\} ~=~
\sum_{i=1}^{k} \, (\rho_i)^n - \sum_{i=1}^{d_P} \,(\beta^{(i)})^n,
\end{equation}
where $(\rho_i)_i$ is the collection of $k$-th roots of unity,
$(k, d_P) = (p+1, m+p+1)$ if $\beta$ is nonsimple with
$d_{\beta}(1) = 0 . t_1 t_2 \ldots t_m (t_{m+1} \ldots t_{m+p+1})^{\omega}$
and $(k, d_P) = (m, m)$ 
if $\beta$
is simple with $d_{\beta}(1)$ of 
length $m$ 
(i.e. $d_{\beta}(1) = 0 . t_1 t_2 \ldots t_m$).
Moreover, $\beta$ is a Perron number since it is 
a Parry number (Lind, \cite{lothaire}):  
hence the asymptotic growth of 
\eqref{nombrePointsFixesRATIONAL} is 
dictated by the geometry and the moduli of the 
beta-conjugates of $\beta$, all
being algebraic integers lying in Solomyak's fractal 
$\Omega$, of modulus
less than or equal to $(1+\sqrt{5})/2$, and by the geometry 
and the moduli of
the Galois conjugates of $\beta$, all being less than
$\beta$, by definition.

Our objective consists in showing that a germ
of curve exists in a neighbourhood of 
the point $(0, 1/\beta)$ in $\mathbb{C}^2$ (this point being the 
origin of this germ)
each time $\beta > 1$ is a real algebraic number,
that is, roughly speaking, a germ of curve located at the
reciprocal $1/\beta$
of the base of numeration
$\beta$. The construction 
of this germ of curve comes from
a (unique) writting of the one-variable
analytic function $f_{\beta}(z)$ 
as a (unique) two-variable analytic function
parametrized by 
$P_{\beta}^{*}(z)$ and $z-1/\beta$, where
$P_{\beta}^{*}(X) = X^{{\rm deg} \beta} P_{\beta}(1/X)$ is the reciprocal
of the minimal polynomial $P_{\beta}(X)$ of $\beta$:
\begin{equation}
\label{rewrittingFBETA} 
f_{\beta}(z) ~=~ G\Bigl(P_{\beta}^{*}(z), z-1/\beta\Bigr),
\end{equation}
where $G = G_{\beta}(U,Z) \in 
\mathbb{C}[[U]][Z]$, 
deg$_{Z}(G_{\beta}(U,Z)) <~ $deg $\beta$,
is convergent,
with coefficients
in $\mathbb{C}$, possibly in some cases in
the algebraic number field 
$\mathbb{K}_{\beta} := \mathbb{Q}(\beta)$, or in a finite 
algebraic extension
of $\mathbb{K}_{\beta}$.

The existence of this germ of curve
arises from the fact that
$\beta > 1$ is a real number which is an algebraic number, since
it is constructed from 
the imposed parametrization $(P_{\beta}^{*}(z), z-1/\beta)$,
which makes use of the
minimal polynomial of $\beta$.
This parametrization
of $G_{\beta}(U,Z)$
leads to the identity  \eqref{rewrittingFBETA}.

Applying the theory of Puiseux \cite{casasalvero} 
\cite{duval}
to \eqref{rewrittingFBETA} 
provides a canonical decomposition of this germ 
into irreducible curves, conjugacy classes, 
as stated in Theorem \ref{puiseuxFBETA}. 
This decomposition brings to light 
several new features of the 
Parry Upper function $f_{\beta}(z)$: 

(i) a new definition of 
the beta-conjugates of $\beta$ in terms
of the Puiseux expansions of the germ
(Definition \ref{BETACONJUGATEpuiseuxFBETAnewdef}), 

(ii) the explicit relations between the 
field of coefficients of the Puiseux series of the
germ $G_{\beta}$, and the beta-conjugates,

(iii) a product formula, as given by \eqref{decccz};
in particular, if
$\beta$ is a Parry number, from
\eqref{fzetaNONsimple} and \eqref{fzetaSIMPLE},
this product gives
an analog of the Euler product
of the Riemann zeta function
for the dynamical zeta function
$\zeta_{\beta}(z)$, where the product is taken over
the different rational
conjugacy classes of the germ (as given by
\eqref{zetaPRODUCT++}).

In addition to the usual Galois conjugation 
relating the roots of the minimal polynomial of
$\beta$, a new conjugation relation, called ``Puiseux-conjugation",
among the beta-conjugates, is defined.

The reader accustomed to numeration systems 
and to the theory of Puiseux
for germs of curves 
can skip Section \ref{S3} and Section \ref{S4} to
proceed directly to beta-conjugates in Section \ref{S5}.

\section{Origin of the work}
\label{S2}

The present note finds its origin in \cite{salem},
for the parametrization by $(P^{*}_{\beta}(z), z-\frac{1}{\beta})$,
and in the two articles
\cite{boyd1} \cite{bertinboyd},
for the idea of developping a two-variable analytic function
canonically associated with the beta-transformation and
the minimal polynomial of the base of numeration $\beta$.
Let us recall them.

In Theorem IV in \cite{salem}, for constructing 
convergent families of
Salem numbers $(\tau_m)_m$ for which the limit is a 
(nonquadratic) Pisot number $\theta$, Salem introduces
polynomials of the following type
\begin{equation}
\label{salemPOLY}
Q_{m}(z) = z^m P_{\theta}(z) + P^{*}_{\theta}(z)
\quad \mbox{or}\quad
Q_{m}(z) = \left(z^m P_{\theta}(z) - P^{*}_{\theta}(z)\right)/(z-1)
\end{equation}
where $Q_{m}(\tau_m) = 0$ and
$P_{\theta}(X)$ is the minimal polynomial of the limit 
$\theta$. We may consider $Q_{m}(z)$ in 
\eqref{salemPOLY}, in one or the other form, as parametrized
by the couple $(P^{*}_{\theta}(z), z)$ (ordered pair). 
This
parametrization, and its consequences, were developped and extended
by Boyd \cite{boyd1} to a more general form, by adding ingeniously
and in a ``profitable" way a second variable $t$, as follows
$$Q(z,t) = z^n P_{\theta}(z) \pm \,t \, z^k \, P^{*}_{\theta}(z)$$
with $n, k$ integers.    
The advantage
of introducing a second variable $t$, as ``continuous parameter",
lies in the fact that an algebraic curve $z = Z(t)$ is associated to
$Q(z,t) = 0$, with a finite number 
of branches and multiple points \cite{boyd3}. 
Boyd \cite{boyd1} shows that the 
existence of this curve gives a deep insight
into the geometry of the roots of $Q(z,t)=0$, for some
values of $t$, in particular those
roots on the unit circle. 
Using these polynomials Bertin and Boyd \cite{bertinboyd} explore
the interlacing of the Galois conjugates of Salem numbers 
with the roots of associated polynomials
(Theorem A and Theorem B in \cite{bertinboyd}).

\section{Functions of the R\'enyi-Parry numeration system in base $\beta > 1$}
\label{S3}
A Salem number is an algebraic integer $> 1$
for which all the Galois conjugates lie
in the closed unit disc,
with at least one conjugate on the unit circle;
the degree of a Salem number 
is even, greater than $4$, and its minimal polynomial is reciprocal
(a Salem number is Galois-conjugated to its inverse) \cite{bertinetal}.
A Perron number is either 1 or
an algebraic integer $\beta > 1$
such that all its Galois conjugates
$\beta^{(i)}$ satisfy:
$|\beta^{(i)}| < \beta$ for $i=1, 2, \ldots, {\rm deg}(\beta) - 1$, if
the degree of $\beta$ is denoted by
deg$(\beta)$ (with $\beta^{(0)} = \beta$).
A Pisot number $\beta$
is a Perron number $\neq 1$ which has the property:
$|\beta^{(i)}| < 1$ for $i = 1, 2, \ldots, {\rm deg}(\beta) - 1$
(with $\beta^{(0)} = \beta$).

Let $\beta > 1$ be a real number and 
define the beta-transformation
$T_{\beta} : [0,1] \to [0,1], x \to \{\beta x\}$ 
($\lceil x \rceil$, resp. $\lfloor x \rfloor$, 
denotes the closest integer to the real number $x$,
$\geq x$, resp. $\leq x$, and
$\{x\}$ its fractional part). 
Denote
$T_{\beta}^{0} =\, $Id, 
$T_{\beta}^{j} = T_{\beta}(T_{\beta}^{j-1})$,
and $t_j = t_{j}(\beta) := 
\lfloor \beta \,T_{\beta}^{j-1}(1) \rfloor$, $j \geq 1$
(the dependency of each $t_j$ to $\beta$ will not be
indicated in the sequel). The digits $t_j$
belong to the finite alphabet
$\mathcal{A}_{\beta} = \{0, 1, \ldots, \lceil \beta - 1 \rceil \}$.
The R\'enyi $\beta$-expansion of unity is denoted by
\begin{equation}
\label{rereDBETA}
d_{\beta}(1) = 0 . t_1 t_2 t_3 \ldots
\qquad \mbox{ and corresponds to }
\qquad 
1 = \sum_{j \geq 1} t_j \beta^{-j}
\end{equation}
obtained by the Greedy algorithm applied to 
$1$ by the successive negative powers
of $\beta$.
The set of successive iterates of $1$ under
$T_{\beta}$, hence the sequence $(t_i)_{i \geq 1}$, 
has the important property that
it controls the admissibility of finite and 
infinite words written in base
$\beta$ over the alphabet $\mathcal{A}_{\beta}$, 
that is the language in base $\beta$,
by the so-called Conditions of Parry 
\cite{frougny1} \cite{lothaire} \cite{vergergaugry2}.

A Parry number $\beta$ is a real number $> 1$
for which the sequence of digits 
$(t_i)_{i \geq 1}$ in the 
R\'enyi $\beta$-expansion of unity
$d_{\beta}(1) = 0 . t_1 t_2 t_3 \ldots$
either ends in infinitely many zeros, in which case 
$d_{\beta}(1)$ is said to be finite and
$\beta$ is said a simple Parry number,
or 
is eventually periodic.
In the second case, if the preperiod length is zero, 
$d_{\beta}(1)$
is said to be purely periodic.
The set of Parry numbers is denoted by
$\mathbb{P}_P$.

Let $\overline{\mathbb{Q}}$ be the set of algebraic numbers.
Denote by $\tb$, resp. $\sb$, resp. $\mathbb{P}$,  the set of Salem numbers,
resp. Pisot numbers, resp. Perron numbers.
After Bertrand-Mathis \cite{bertrandmathis},
Schmidt \cite{schmidt}, Lind \cite{lothaire}, 
the following inclusions hold
$$\sb ~\subset~ \mathbb{P}_P ~\subset~ \mathbb{P} ~\subset~ 
\overline{\mathbb{Q}}.$$
The question of the dichotomy 
$\mathbb{P} = \mathbb{P}_P \, \cup \, (\mathbb{P} \setminus \mathbb{P}_P)$
is an important open question, which 
amounts to finding a method for
discrimating when a Perron number $> 1$ 
is a Parry number or not.
In particular, for Salem numbers, 
though conjectured to be nonempty with
a positive density \cite{boyd3},
the set $\tb \setminus \mathbb{P}_P$ 
is not charaterized yet. For now, it is a fact that 
all the small Salem numbers, 
for instance those given by Lehmer 
in \cite{lehmer}, and many others known, 
are Parry numbers \cite{boyd1} \cite{boyd3}.
The set of simple Parry numbers contains 
$\mathbb{N} \setminus \{0, 1\}$ and is dense in
$(1, +\infty)$ \cite{parry}.

Let $\beta$ be a Parry number, with
$d_{\beta}(1) = 0 . t_1 t_2 \ldots t_m (t_{m+1} \ldots t_{m+p+1})^{\omega}$.
If $ m \neq 0$, the integer $m$ is the preperiod length of $d_{\beta}(1)$;
if $p+1 \geq 1$, the period length of $d_{\beta}(1)$
is $p+1$. 
The iterates of $1$ under $T_{\beta}$
are polynomials: 
$T_{\beta}^{n}(1) = \beta^n - t_1 \beta^{n-1} -t_2 \beta^{n-2} \ldots -t_{n}$
(by induction). This observation allows Boyd
in \cite{boyd2} to define uniquely the Parry polynomial of $\beta$.
Indeed, writting $r_{n}(X)= X^n - t_1 X^{n-1} -t_2 X^{n-2} \ldots -t_{n}$, we have
$r_{n}(\beta) = T_{\beta}^{n}(1)$ and 
$\beta$ satisfies the polynomial equation $P_{\beta,P}(\beta)=0$, where
\begin{equation}
\label{defiParryPOL_BOYD} 
P_{\beta,P}(X) :=
\left\{
\begin{array}{ll}
r_{m+p+1}(X) - r_{m}(X) 
& \mbox{if}~ m > 0 \, ~(p+1 \geq 1),\\
r_{p+1}(X) - 1
& \mbox{if}~ m = 0 \, ~(p+1 \geq 1, \mbox{``purely periodic"}),\\
r_{m}(X)
& \mbox{if}~ m \geq 1 \, ~(p+1 = 0, \mbox{``simple"}).
\end{array}
\right.
\end{equation}
The Parry polynomial $P_{\beta,P}(X)$ of the Parry number
$\beta$, monic, of degree
$d_P = m+p+1$, multiple of the minimal polynomial
$P_{\beta}(X)$ of $\beta$, can also be defined
from the rational fraction $\zeta_{\beta}(X)$: 
its reciprocal $P_{\beta,P}^{*}(z)$, 
of the complex variable $z$,
is the denominator of the meromorphic function
$\zeta_{\beta}(z)$, given in both cases by 
\eqref{fzetaNONsimple} and 
\eqref{fzetaSIMPLE} (``simple" case).
Boyd \cite{boyd2} defines the beta-conjugates of $\beta$
as being the roots
of $P_{\beta,P}(X)$, canonically attached to $\beta$,
which are not the Galois conjugates of $\beta$.
Beta-conjugates are algebraic integers.

For any real number $\beta > 1$,
from the sequence $(t_i = t_{i}(\beta))_{i \geq 1}$ we form
the Parry Upper function $f_{\beta}(z) := -1 + \sum_{i \geq 1} t_i z^i$
at $\beta$, of the complex variable $z$.
The terminology ``Parry Upper" comes from
the fact that $(t_i)_{i \geq 1}$ gives the 
upper bound for admissible words in base $\beta$,
where being lexicographically smaller than this upper bound, with all its shifts,
means satisfying the Conditions of Parry
for admissibility \cite{frougny1} \cite{lothaire}
\cite{vergergaugry2}. 

When $\beta$ is a Parry number, 
the inverses $\xi^{-1}$ of the
zeros $\xi$ of
the analytic function
$f_{\beta}(z)$ are exactly the roots of the Parry polynomial 
$P_{\beta,P}(X)$ of $\beta$
(from \eqref{fzetaNONsimple}, \eqref{fzetaSIMPLE};
\cite{vergergaugry2}). 
In particular we have 
$f_{\beta}(1/\beta)=0$ by \eqref{rereDBETA}. 
The multiplicity of the root
$1/\beta$ in $f_{\beta}(z)$ is one by the fact that 
$f'_{\beta}(1/\beta) = \sum_{i \geq 1} i t_i \beta^{i-1} > 0$.
Hence in the factorization of $P_{\beta,P}(X)$ the multiplicity
of the minimal polynomial $P_{\beta}(X)$ of $\beta$ is one.
But the determination of 
the multiplicity of a beta-conjugate of $\beta$ and of the
factorization of the Parry polynomial of $\beta$
is an open problem \cite{boyd2} \cite{vergergaugry2}. 
We give
a partial solution to this problem
by showing how this factorization can be deduced from
the germ of curve ``at $1/\beta$" and the theory of Puiseux.

Though the degree $d_P$ of the Parry polynomial
$P_{\beta,P}(X)$ of a Parry number $\beta$
be somehow an obscure function of $\beta$, 
the Parry polynomial $P_{\beta,P}(X)$, say
$=  \sum_{i=0}^{d_P} a_i X^i$,
has the big advantage, as compared to the minimal polynomial
$P_{\beta}(X)$ of $\beta$, to exhibit
a naive height
H$(P_{\beta,P}) = \max_{i=0,1,\ldots,d_P} |a_i|$
in $\{\lfloor \beta \rfloor, \lceil \beta \rceil\}$ \cite{vergergaugry2}.
This control of the height by the base of numeration
$\beta$ has an important consequence: 
given a convergent family of Parry numbers
$(\beta_j)_j$, 
an Equidistribution Limit Theorem for 
the conjugates $(\beta^{(i)}_{j})_{i,j}$ holds
with a limit measure 
which is the Haar measure on the unit circle
\cite{vergergaugry2}, under some assumptions.
Solomyak's fractal $\Omega$ is densely occupied by all the
conjugates of all the Parry numbers \cite{solomyak}, with
a major concentration of conjugates occuring in a
neighbourhood of the unit circle.

Beta-conjugates are then equivalently defined either
as roots of $P_{\beta,P}(X)$, as inverses of zeros of 
$f_{\beta}(z)$, as inverses of poles of  
the dynamical zeta function $\zeta_{\beta}(z)$.
The three equivalent definitions arise from the relations
\eqref{fzetaNONsimple} and \eqref{fzetaSIMPLE} 
(``simple" case), deduced from \cite{itotakahashi} 
\cite{flattolagariaspoonen}.

The Galois- and beta- conjugates $\beta^{(i)}$
of a Parry number $\beta$ all lie in
Solomyak's fractal \cite{solomyak}, 
represented in Figure \ref{fractalSOLO}.
The left extremity of the spike on the real 
negative axis is $-(1+\sqrt{5})/2$ and
the general bound 
$|\beta^{(i)}| \leq  (1+\sqrt{5})/2$ holds
for all $i$ and all Parry numbers $\beta$;
this upper bound was also found by 
Flatto, Lagarias and Poonen \cite{flattolagariaspoonen}.

\begin{figure}
\begin{center}
\includegraphics[width=7cm]{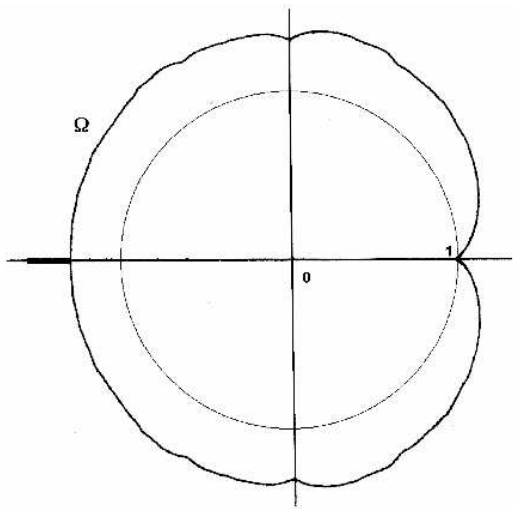}
\end{center}
\caption{Solomyak's fractal $\Omega$.}
\label{fractalSOLO}
\end{figure}

Let $\beta$ be a Parry number. The three following 
assertions are obviously equivalent:
(i) $\beta$ has no beta-conjugate, (ii) the Parry polynomial
of $\beta$ is irreducible, (iii) the Parry polynomial of
$\beta$ is equal to the minimal polynomial of $\beta$.
For some families of Parry numbers 
\cite{kwon} \cite{vergergaugry2} it is possible
to deduce the irreducibility of their Parry polynomials.

By Szeg\H{o}-Carlson-Poly\'a Theorem \cite{dienes}, 
the Parry Upper function
$f_{\beta}(z)$ is a rational fraction if and only if
$\beta$ is a Parry number \cite{vergergaugry2}.
If $\beta > 1$ is an algebraic number, but 
not a Parry number, $f_{\beta}(z)$
is an analytic function on the open unit disc
with the unit circle as natural boundary.

For $\beta > 1$ any algebraic number, 
except a Parry number,
we define a beta-conjugate of $\beta$
as the inverse of a zero of the function
$f_{\beta}(z)$, if it exists. 
A priori, it may happen that
$f_{\beta}(z)$ admits the only zero 
$1/\beta$ in its domain of definition
$D(0,1)$, with $|z|=1$ as natural
boundary. The problem of the existence of zeros
of $f_{\beta}(z)$ in $D(0,1)$ 
is linked to the gappiness (the terminology 
{\it gappiness} was
introduced in \cite{vergergaugry0}
as a notion which is much weaker than that of {\it lacunarity};
indeed {\it lacunarity} 
is classically associated to Hadamard gaps)  
of the sequence
$(t_i)$ and its Diophantine approximation properties
\cite{vergergaugry0} \cite{adamczewskibugeaud}; 
this gappiness cannot be too large at infinity
and the Ostrowski ``quotients of the gaps"
are dominated by $\log {\rm M}(\beta)/\log \beta$, 
where M$(\beta)$ is the Mahler measure of $\beta$.

By a Theorem of Fuchs \cite{vergergaugry1},
if $f_{\beta}(z)$ is such that
$(t_i)$ admits Hadamard gaps, then the number of zeros of
$f_{\beta}(z)$ is infinite in $D(0,1)$. 
This occurence, of having Hadamard gaps, is 
conjectured to be true for infinitely many
transcendental numbers $\beta > 1$ but to be impossible
as soon as $\beta > 1$ is an algebraic number.
If $\beta > 1$ is an algebraic number, 
the number of zeros of $f_{\beta}(z)$ in $D(0,1)$, 
i.e. the number of beta-conjugates
of $\beta$ of modulus $> 1$,
is conjectured to be finite.
This finiteness property of the number of beta-conjugates
would be in agreement with
the existence of an integer $M \geq 1$ in
\eqref{nombrePointsFixes}, in the context of the 
dynamical zeta function.

\section{Fractionary power series and Puiseux expansions for germs of curves}
\label{S4}

In the sequel, 
we will follow Casas-Alvero \cite{casasalvero},
Duval \cite{duval}, Walker \cite{walker},
Walsh \cite{walsh} and restrict ourselves to
what is needed for the application of the
theory of Puiseux to beta-conjugates of algebraic numbers $> 1$, 
to fix notations.  
The terminology "fractionary" is taken from
\cite{casasalvero}.
Let $k$ be a (commutative) field of characteristic zero and
let $G(X,Y) \in k[[X,Y]]$. 
We consider the formal equation
$$G(X,Y) = 0$$
and 
are interested in solving
it for $Y$, that is 
we want to find some sort of series
in $X$, say $Y(X)$, with coefficients in $k$, such that
\begin{equation}
\label{Gequation}
G(X, Y(X)) = 0,
\end{equation}
$G(X,Y(X))$ being the series in $X$ obtained by substituting
$Y(X)$ for $Y$ in $G$.
The series $Y(X)$ is called a $Y$-root of $G$.
When $k = \mathbb{C}$, this general 
problem was considered by Newton.
In the following we will consider
$k = \mathbb{C}$ and will consider
rationality questions over 
smaller fields $k$ in Section \ref{S6}.
 
For solving \eqref{Gequation}, 
we need to deal with series in fractionary powers of $X$.
First, let us define the field of 
fractionary power series over $\mathbb{C}$.
Denote $\mathbb{C}((X))$  the field of the formal Laurent series
$$\sum_{i=d}^{\infty} a_i X^i , \qquad d \in \mathbb{Z}, a_i \in \mathbb{C}.$$

An element of $\mathbb{C}((X^{1/n}))$ has the form
$$s = \sum_{i \geq r}  a_i X^{i/n}.$$

The field of fractionary power series is
denoted by $\mathbb{C}\!\!\!\ll \!\!X \!\!\!\gg$ 
and by definition is the direct limit
of the system
$$\Bigl\{ \mathbb{C}((X^{1/n})), \iota_{n,n'} \Bigr\},$$
where, for $n$ dividing $n'$ (with $n' = d n$),
$$\iota_{n,n'} : \mathbb{C}((X^{1/n})) \to \mathbb{C}((X^{1/n'})),\quad
\sum  a_i X^{i/n} ~\to~ \sum  a_i X^{d  i/ d n}.$$

A Puiseux series is by definition a fractionary power series
$$s = \sum_{i \geq r}  a_i X^{i/n}$$
for which the order in $X$ 
$$o_{X}(s) := \frac{\min\{i \mid a_i \neq 0\}}{n}$$
is (strictly) positive. 
A natural representant of its class in the direct limit
is such that
$n$ and gcd$\{i \mid a_i \neq 0\}$ have no common factor;
then 
$n$ is called the ramification index 
(or polydromy order) of $s$, denoted by $\nu(s)$.

If $s \in \mathbb{C}((X^{1/n}))$ is a Puiseux series, with
$n = \nu(s)$ its ramification index, the series
$\sigma_{\epsilon}(s), \epsilon^n = 1$, will be called the
conjugates of $s$, where
$$\sigma_{\epsilon}(s) = \sum_{i \geq r}  \epsilon^i a_i X^{i/n}.$$
The set of all (distinct) 
conjugates of $s$ is called the conjugacy class of
$s$. The number of different conjugates 
of $s$ is $\nu(s)$.

Let us recall the Newton polygon of a two-variable formal series. Let
$$G = G(X,Y) =\sum_{i > 0, j > 0} A_{i,j} X^{i} Y^j \qquad \in \mathbb{C}[[X,Y]] $$
and obtain the discrete set of points 
with nonnegative integral coefficients
$$\Delta(G) := \{(i,j) \mid A_{i,j} \neq 0 \},$$
called the Newton diagram of $G$. Let 
$(\mathbb{R}^{+})^2 := \{(x,y) \mid x \geq 0, y \geq 0\}$ 
be the first quadrant in the plane $\mathbb{R}^2$ 
and consider
$$\Delta'(G) := \Delta(G) + (\mathbb{R}^{+})^2 .$$
Then the convex hull of $\Delta'(G)$ 
admits a border which is composed
of two half-lines (a vertical one, an horizontal one,  
parallel to the coordinate axes) 
and a polygonal line, called the Newton polygon of
$G$, joining them, denoted by $\mathcal{N}(G)$. 
The height $h(\mathcal{N}(G))$ of $G$ 
is by definition the maximal ordinate of the vertices
of the Newton polygon $\mathcal{N}(G)$.

If $y(X) = \sum_{q \geq 1} a_q \left(X^{1/\nu(y)}\right)^{q}$ 
is a Puiseux series, write 
$G_y = G_{y}(X,Y) =\prod_{i=1}^{\nu(y)} (Y - y_{i}(X))$,
the $y_{i}, i= 1, \ldots, 
\nu(y)$ being the conjugates of $y$.
The series $G_y$ is irreducible in 
$\mathbb{C}[[X,Y]]$.
The theory of Puiseux allows a formal decomposition as follows.

\begin{theorem}
\label{decompoMAINthm}
For any $G = G(X,Y) \in \mathbb{C}[[X,Y]]$,
\begin{itemize}
\item[(i)] there are Puiseux series 
$y_{1}, y_{2}, \ldots, y_{m}, m \geq 0$, in
$\mathbb{C}\!\!\ll \!\!X \!\!\gg$
so that $G$ decomposes in the form
$$G = u \, X^{r} \, G_{y_1} \, G_{y_2} \ldots G_{y_s}
$$
where $r \in \mathbb{Z}$, and $u$ is an invertible series
in $\mathbb{C}[[X,Y]]$,
\item[(ii)] the height of the Newton polygon
of $g$ is the sum of the ramification indices 
$$h(\mathcal{N}(G)) = \nu(y_1) + \nu(y_2) + \ldots + \nu(y_s)$$
and the $Y$-roots of $G$ are the conjugates of the
$y_{j}(X), j= 1, \ldots ,s$.
\end{itemize}
\end{theorem}

The Newton-Puiseux algorithm applied to the Newton polygon
$\mathcal{N}(G)$ of $G$ allows to compute all the $Y$-roots
of $G(X,Y)$ and the ramification indices 
\cite{casasalvero} \cite{duval} \cite{walker}.

\begin{definition}
Let $k$ be a  
(commutative) field 
of characteristic zero
and $g(X,Y) \neq 0$ an element of
$k[[X,Y]]$ such that 
$g(0,0)=0$.
A parametrization of $g$ is an ordered pair 
$(\mu_1(T), \mu_2(T))$ of elements of $k[[T]]$
which satisfies
\begin{itemize}
\item[(i)] $\mu_1$ and $\mu_2$ are not simultaneously identically zero,
\item[(ii)] $\mu_1(0)= \mu_2(0) = 0$,
\item[(iii)] $g(\mu_1(T), \mu_2(T))=0 \, \in k[[T]]$.
\end{itemize}
\end{definition}

Denote $\mathbb{C}\{x_1, x_2, \ldots, x_q\}$ the ring 
of convergent power series, and turn to convergence questions.
Let $s = \sum_{i \geq 0} a_i X^{i/n}$ be a fractionary power series,
with $a_i \in \mathbb{C}$. We say that
$s$ is a convergent fractionary power series if and only if
the ordinary power series
$$s(t^n) = \sum_{i \geq 0} a_i t^i$$
has nonzero convergence radius. 
This condition does not depend upon the integer
$n$ and the set of convergent fractionary power series 
$\mathbb{C}\{X\}$ is
a subring of $\mathbb{C}\!\!\ll\!X\!\gg$.

If $s$ is convergent, with $\nu(s) =n$,
one may compose the polydromic
(multivalued) function $z \to z^{1/n}$ 
and the analytic function
defined by $s(t^n)$ in a neighbourhood of 
$t=0$: we obtain a polydromic function
$\overline{s}$, defined in a neighbourhood of 
$z=0$, which we call the (polydromic) function 
associated with $s$. 
If $s$ is convergent, all its conjugates
are also convergent and any of them 
defines the same polydromic function 
$\overline{s}$ as $s$. 
If $s$ is convergent, the associated
function $\overline{s}$ takes $\nu(s)$
different values on each $z_0 \neq 0$
in a suitable neighbourhood
of $0$.

In the context of convergent series the theory of Puiseux
makes Theorem \ref{decompoMAINthm} more accurate as follows.

\begin{theorem}
\label{decompCONVERGENTthm}
If $G(x,y) \in \mathbb{C}\{x, y\}$ 
is a convergent series, then all its $y$-roots
are convergent, and there are an invertible
series $v \in \mathbb{C}\{x,y\}$ and a nonnegative integer
$r$, both uniquely determined  
by $G$, and convergent Puiseux series
$y_1, y_2, \ldots, y_s$, uniquely 
determined by $G$ up to conjugation
so that
\begin{equation}
\label{GgermCONVERGENT}
G = v x^r G_{y_1} G_{y_2} \ldots G_{y_s}.
\end{equation}
\end{theorem}

If $G$ is a polynomial in $Y$, i.e. if 
$G \in \mathbb{C}[[X]][Y]$, and
if the coefficients $a_q$ of the Puiseux expansions involved in its decompositon
are algebraic numbers,
denote $L = \mathbb{Q}(a_1, a_2, \ldots)$ the number field generated
by the coefficients. Assume 
$[L : \mathbb{Q}] < +\infty$ and
let $r := [L : \mathbb{Q}]$. 
Let $\sigma_1, \sigma_2, \ldots, \sigma_r$, the
$r$ embeddings of $L$ into $\overline{\mathbb{Q}}$.
Denote
$$C = C(y(X)) := \left\{
\sum_{q \geq 1} \sigma_{i}(a_q) \left(
\zeta_{\nu(y)}^{j} X^{1/\nu(y)} \right) 
\mid
i = 1, \ldots,r \, , \, 
j= 0, 1, \ldots, \nu(y)
\right\}$$
the $L$-rational conjugacy class of $y(X)$.
By Proposition 2.1 in Walsh \cite{walsh}, 
assuming that all the Puiseux expansions of $X$ 
in $G$ are distinct,
$$\prod_{i=1}^{\nu(y)} (Y - y_{i}(X))$$
is irreducible in $\overline{\mathbb{Q}}((X))[Y]$, of degree 
$\nu(y)$ in $Y$, and
\begin{equation}
\label{produitPARclassesRATIONAL}
\prod_{y_i \in C} (Y - y_{i}(X))
\end{equation}
is irreducible in $\mathbb{Q}((X))[Y]$ of degree
$\nu(y) r/r_0$ in $Y$ where
$$r_0 := \{\sigma : L \to \overline{\mathbb{Q}}
\mid
\exists t \in \mathbb{Z} ~\mbox{such that}~
\sigma(a_q) = a_q \zeta_{\nu(y)}^{t q}
~\mbox{for all}~ q \geq 1
\}.$$
If, in addition, $G$ is assumed convergent, gathering the
Puiseux expansions by $L$-rational conjugacy classes, whose
number is (say) $e$, the collection of such classes
being 
$(C_j)_{j=1,\ldots,e}$, allows to write
$G$ in the form of the product of a unit 
$v \in \mathbb{C}[[x,y]]$
by a nonnegative power $x^r$
of the first variable $x$ and 
a product of $e$ irreducible polynomials
in $\mathbb{Q}[[x]][y]$ as follows:

\begin{equation}
\label{GgermCONVERGENT}
G = v x^r \prod_{j=1}^{e} \prod_{y_i \in C_j} (y - y_{i}(x)).
\end{equation}

\section{Beta-conjugates as Puiseux expansions}
\label{S5}

Let $\beta > 1$ be an algebraic number, 
not necessarily a Parry number. 
In the sequel we will not consider 
the case where $\beta > 1$ is a rational integer:
indeed, in this case, 
$\beta$ has no Galois conjugate $\neq \beta$, and
$f_{\beta}(z)=-1 + \beta z$ is a polynomial
having only the root $1/\beta$; therefore $\beta$ 
has no beta-conjugate.

The
key observation, that the three 
functions
$z-1/\beta$,
$P_{\beta}^{*}(z)$,
$f_{\beta}(z)$  
cancel at $1/\beta$, each of them with multiplicity one,
leads to consider the 
point $(0, 1/\beta)$ of $\mathbb{C}^{2}$ as natural origin 
of the germ of curve.
Therefore we consider the new variable $Z:=z-1/\beta$ and
make the change of variable $z \to Z$ into $f_{\beta}(z)$ and
$P_{\beta}^{*}(z)$, as follows:
$$\widetilde{f_{\beta}}(Z) := f_{\beta}(z), \qquad
\widetilde{P_{\beta}^{*}}(Z) := P_{\beta}^{*}(z).$$

\begin{lemma}
\label{fbetatilde}
Let $\beta > 1$ be a real number. Then
\begin{equation}
\label{exprfbetatilde}
\widetilde{f_{\beta}}(Z) ~=~ \sum_{j \geq 1}
\lambda_j Z^j
\end{equation}
with
$\lambda_j = \lambda_{j}(\beta) := 
\sum_{q \geq 0} t_{j+q}
\left(
\begin{array}{cc}
j+q\\j
\end{array}
\right)
\left(\frac{1}{\beta}\right)^q .
$
\end{lemma}

\begin{proof}
Expanding $f_{\beta}(z) 
= -1 + \sum_{i \geq 1} t_i (z - \frac{1}{\beta} 
+ \frac{1}{\beta})^i$ as a function of $Z=z-1/\beta$ readily gives
\eqref{exprfbetatilde}.
\end{proof}

Let $\beta > 1$ be any real number.
The series $\lambda_j = \lambda_{j}(\beta), j \geq 1,$ 
have nonegative terms and,
by Stirling's formula applied to the binomial coefficients,
are convergent.

\begin{proposition}
\label{lambda_jCONTINUOUS}
Let $\beta > 1$ be a real number. 
For all $j \geq 1$, 
the map $(1,+\infty) \to \mathbb{R}^{+} , 
\beta \to \lambda_{j}(\beta)$
is right-continuous. The set of discontinuity points is contained
in the set of simple Parry numbers.
\end{proposition}

\begin{proof}
Assume $\beta > 1$ a real number which is not an integer.
Let us fix $j \geq 1$. There exists $u > 0$
such that the open interval
$(\beta - u, \beta + u)$ contains no integer. Then 
any $\beta' \in (\beta - u, \beta + u)$ is such that
its R\'enyi $\beta'$-expansion $d_{\beta'}(1)$ of $1$ 
has digits $t_{q}(\beta')$ within the same alphabet
which is $\mathcal{A}_{\beta} =
\{0, 1, \ldots, \lfloor \beta \rfloor\}$.
Let $\epsilon > 0$. Then there exists
$q_0 \geq j$ such that
$$\sum_{q > q_0}
\left(
\begin{array}{cc}
q\\j
\end{array}
\right) 
\left(
\frac{1}{\beta - u}
\right)^{q-j} < \frac{\epsilon}{4 \lfloor \beta \rfloor}.
$$
Then, for all $\beta' \in (\beta - u, \beta + u)$,
since $1/\beta' \leq 1/(\beta - u)$,
the following uniform inequality holds:
\begin{equation}
\label{follo}
\sum_{q > q_0}
t_{q}(\beta') 
\left(
\begin{array}{cc}
q\\j
\end{array}
\right)
\left(
\frac{1}{\beta'}
\right)^{q-j} < \frac{\epsilon}{4}.
\end{equation}
Now there are are two cases: either $\beta$ is a simple Parry number, or not.

(i) Assume $\beta > 1$ is not a simple Parry number.
Then the sequence $(t_{i}(\beta))_i$ is infinite 
(does not end in infinitely many zeros).
There exists $\eta > 0, \eta < u,$ 
small enough such that
$t_{1}(\beta') = t_{1}(\beta),
t_{2}(\beta') = t_{2}(\beta),
\ldots,
t_{q_1}(\beta') = t_{q_1}(\beta)$
for all $\beta' \in 
(\beta - \eta, \beta + \eta)$
with $q_1 = q_{1}(\beta') > q_{0}, 
t_{q_{1}+1}(\beta') \neq t_{q_{1}+1}(\beta)$, 
for which, since 
$\beta' \to \beta'^{q-j}, q= j, j+1, \ldots, q_0$, 
are all continuous,
\begin{equation}
\label{bubu}
\left|
\sum_{q=j}^{q_0} t_{q}(\beta)
\left(
\begin{array}{cc}
q\\j
\end{array}
\right)
\left(
\left(
\frac{1}{\beta'}
\right)^{q-j}
-
\left(
\frac{1}{\beta}
\right)^{q-j}
\right)
\right|
< \epsilon/2.
\end{equation}
In this nonsimple Parry case, recall \cite{parry} 
that the function
$\beta' \to q_{1}(\beta')$ is 
monotone increasing and locally constant when
the variable $\beta'$ tends to $\beta$
(i.e. $d_{\beta'}(1)$ and
$d_{\beta}(1)$ start by the same string of digits
$t_1 t_2 \ldots t_{q_1}$ when $\beta'$ is close to $\beta$).

(ii) Assume that $\beta > 1$ is a simple Parry number.
Let $d_{\beta}(1) = 0 . t_1 t_2 \ldots t_N$ be its R\'enyi 
$\beta$-expansion of unity ($N \geq 1$).
If $N > q_0$, there exists 
$\eta > 0$, $\eta < u$,
such that $|\beta' - \beta| < \eta 
\Longrightarrow
t_{q}(\beta') = t_{q}(\beta)
$ for all $q = 1, \ldots, N-1$,
and \eqref{bubu} also holds. 
If $j \leq N \leq q_0$, we express $\beta$ in base $\beta$ and
$\beta'$ in base $\beta'$ in the sense of R\'enyi:
then we deduce that there exists 
$\eta > 0$, $\eta < u$,
such that 
$\beta \leq \beta' < \beta + \eta$
implies
$$\left|
\sum_{q=N+1}^{q_0} t_{q}(\beta')
\left(
\frac{1}{\beta'}
\right)^{q-j}
\right|
< \frac{\epsilon}{4} \frac{1}{\max_{q=N+1, \ldots, q_0}\{
\left(
\begin{array}{cc}
q\\j
\end{array}
\right)
\}}$$
and
\begin{equation}
\label{inn2}
\left|
\sum_{q=j}^{N} t_{q}(\beta)
\left(
\begin{array}{cc}
q\\j
\end{array}
\right)
\left(
\left(
\frac{1}{\beta'}
\right)^{q-j}
-
\left(
\frac{1}{\beta}
\right)^{q-j}
\right)
\right|
< \epsilon/4 ;
\end{equation}
in this case,
\begin{equation}
\label{inn1}
\left|
\sum_{q=N+1}^{q_0} t_{q}(\beta')
\left(
\begin{array}{cc}
q\\j
\end{array}
\right)
\left(
\frac{1}{\beta'}
\right)^{q-j}
\right|
< \epsilon/4.
\end{equation}
If $q_0 \leq N$, we deduce, for all $\beta' \in
(\beta, \beta + \eta)$,
$$
\left|
\lambda_{j}(\beta) - \lambda_{j}(\beta') 
\right| 
\leq
\left|
\sum_{q=j}^{q_0} t_{q}(\beta)
\left(
\begin{array}{cc}
q\\j
\end{array}
\right)
\left(
\left(
\frac{1}{\beta'}
\right)^{q-j}
-
\left(
\frac{1}{\beta}
\right)^{q-j}
\right)
\right|
+
$$
\begin{equation}
\label{inntot}
\left|
\sum_{q > q_0} 
t_{q}(\beta')
\left(
\begin{array}{cc}
q\\j
\end{array}
\right)
\left(
\frac{1}{\beta'}
\right)^{q-j}
-
\sum_{q > q_0}
t_{q}(\beta)
\left(
\begin{array}{cc}
q\\j
\end{array}
\right)
\left(
\frac{1}{\beta}
\right)^{q-j}
\right|
< \epsilon/2 
+ 2 \epsilon/4 = \epsilon ,
\end{equation}
and, in the case
$j \leq N \leq q_0$, 
we decompose the sum
$\sum_{q=j}^{q_0}$ as
$\sum_{q=j}^{N} \, + \sum_{q=N+1}^{q_0} \,$
in the upper bound \eqref{inntot}, using
\eqref{inn2} and \eqref{inn1}, to obtain
$\left|
\lambda_{j}(\beta) - \lambda_{j}(\beta')
\right| < \epsilon$ as well.
If $j > N$, then $\lambda_{j}(\beta) = 0$;
there exists
$\eta > 0$, $\eta < u$,
such that
$\beta \leq \beta' < \beta + \eta$
implies
\begin{equation}
\label{follo+}
\left|
\sum_{q=j}^{q_0} t_{q}(\beta')
\left(
\frac{1}{\beta'}
\right)^{q-j}
\right|
< \frac{3 \epsilon}{4} \frac{1}{\max_{q=j, \ldots, q_0}\{
\left(
\begin{array}{cc}
q\\j
\end{array}
\right)
\}}.
\end{equation}
Hence, using \eqref{follo} and \eqref{follo+},
for $\beta \leq \beta' < \beta + u$,
$$
\left|
\lambda_{j}(\beta')
\right|
\leq
\left|
\sum_{q=j}^{q_0} t_{q}(\beta')
\left(
\begin{array}{cc}
q\\j
\end{array}
\right)
\left(
\frac{1}{\beta'}
\right)^{q-j}
\right|
+
\left|
\sum_{q > q_0}
t_{q}(\beta')
\left(
\begin{array}{cc}
q\\j
\end{array}
\right)
\left(
\frac{1}{\beta'}
\right)^{q-j}
\right| < \frac{3 \epsilon}{4} + \frac{\epsilon}{4} = \epsilon 
$$
and the right-continuity 
$\lim_{\beta' \to \beta^{+}} \lambda_{j}(\beta') = 0$
for $j > N$.
 
Let us now assume that $\beta > 1$ is an integer.
Then $d_{\beta}(1) = 0 . \beta$, 
$t_{1}(\beta) = \beta, \lambda_{1}(\beta) = \beta$
and
$t_{j}(\beta)=0, \lambda_{j}(\beta) = 0$ for $j \geq 2$.
The same arguments as in (ii), with $N = 1$, 
lead to the result.
\end{proof}

\begin{lemma}
\label{Pbetatilde}
If $\beta > 1$ is an algebraic number of minimal polynomial
$P_{\beta}(X) = a_0 +a_1 X + a_2 X^2 +\ldots + a_d X^d$,
$a_i \in \mathbb{Z}, a_0 a_d \neq 0$, 
then
\begin{equation}
\label{exprPbetatilde}
\widetilde{P_{\beta}^{*}}(Z) = 
\gamma_1 Z + \gamma_2 Z^2 + \ldots + \gamma_d Z^d ,
\end{equation}
with $\gamma_q = 
\sum_{j=q}^{d} a_{d-j} \left(
\begin{array}{cc}
j\\q
\end{array}
\right)
\left(\frac{1}{\beta}\right)^{j-q} 
\, \in
\mathbb{K}_{\beta}, 
\gamma_d = a_0 \neq 0, 
\gamma_1 = P_{\beta}^{{*}'}(1/\beta) \neq 0$.
\end{lemma}

\begin{proof}
The relation 
$\widetilde{P_{\beta}^{*}}(Z) = 
P_{\beta}^{*}(z-\frac{1}{\beta} + \frac{1}{\beta})$
leads to 
$$\widetilde{P_{\beta}^{*}}(Z) = 
\sum_{j=0}^{d} \sum_{q=0}^{j} a_{d-j} \left(
\begin{array}{cc}
j\\q
\end{array}
\right) 
\left(\frac{1}{\beta}\right)^{j-q}
Z^q
~=~ \sum_{q=0}^{d} \sum_{j=q}^{d} a_{d-j} \left(
\begin{array}{cc}
j\\q
\end{array}
\right)
\left(\frac{1}{\beta}\right)^{j-q}
Z^q.$$
The constant term is zero
since $P_{\beta}(\beta)= \sum_{j=0}^{d} a_j \beta^j = 0$.
\end{proof}

\begin{theorem}
\label{existGERMthm}
Let $\beta > 1$ be an algebraic number and 
$P_{\beta}(X)$ its minimal polynomial. Then
there exists a unique polynomial $G = G_{\beta}(U, Z)
\in \mathbb{C}[[U]][Z]$ in $Z$, {\rm deg}$_Z G < \, ${\rm deg} $\beta$, 
such that $( \widetilde{P_{\beta}^{*}}(Z), Z )$ is a parametrization of
$G - \widetilde{f_{\beta}}
\in \mathbb{C}[[U,Z]]$, i.e. such that
\begin{equation}
\label{eqexistGERM}
G_{\beta}(\widetilde{P_{\beta}^{*}}(Z),Z) - \widetilde{f_{\beta}}(Z) ~=~ 0.
\end{equation}
\end{theorem}

\begin{proof}
Uniqueness. 
Assume that $G^{(1)}$ and $G^{(2)}$
are such that $G^{(1)} - \widetilde{f_{\beta}}$ 
and $G^{(2)}- \widetilde{f_{\beta}}$ are both parametrized
by $(\widetilde{P_{\beta}^{*}}(Z), Z)$. Then
$(G^{(1)}-G^{(2)})(\widetilde{P_{\beta}^{*}}(Z),Z)=0$
with
$G^{(1)}-G^{(2)} \in \mathbb{C}[[U]][Z]$,
deg$_Z \,(G^{(1)}-G^{(2)}) < d$.
Assume $G^{(1)} \neq G^{(2)}$
and  
$G^{(1)}-G^{(2)}$ irreducible in $Z$ 
(no loss of generality).
Then this equation defines a plane curve
$$\mathcal{C}_{\beta} := 
\{(u,z) \in \mathbb{C}^{2} \mid 
(G^{(1)}-G^{(2)})(u,z) = 0 \}
$$
along with a ramified covering  
$\pi : \mathcal{C}_{\beta} \to \mathbb{C}$ of the complex plane.
Above all but finitely many points $u$ of
the $U$-plane, the fiber $\pi^{-1}(u)$ has
cardinality $\leq d-1$.
The implicit function theorem
states that there exist $\delta$ analytic functions
$z_{1}(u), \ldots, z_{\delta}(u)$,
$\delta \leq d-1$, such that
$\pi^{-1}(u) = \{z_{i}(u) \mid i = 1, \ldots, \delta\}$ 
and
$(G^{(1)} - G^{(2)})(u, z_{i}(u)) = 0$ for 
$i = 1, \ldots, \delta$.
Each of them parametrizes one sheet 
of the covering in a neighbourhood
of $u$. The contradiction
comes from the fact that
the polynomial $P_{\beta}^{*}(z)$
is irreducible, of degree $d$, that
the parametrization
$(\widetilde{P_{\beta}^{*}}(Z),Z)$
is imposed. Therefore the number of sheets
$\delta$ should be
equal to $d$. 
Contradiction.

Existence: by construction.
Let $U := \widetilde{P_{\beta}^{*}}(Z)$. 
From \eqref{exprPbetatilde}, 
$$U =\gamma_1 Z + \gamma_2 Z^2 + \ldots + \gamma_d Z^d
\Rightarrow
Z^d = \frac{1}{\gamma_d} U - 
\left(
\frac{\gamma_1}{\gamma_d} Z + \frac{\gamma_2}{\gamma_d} Z^2 + \ldots + 
\frac{\gamma_{d-1}}{\gamma_d} Z^{d-1}
\right).
$$
It follows that 
$Z^d \in \mathbb{K}_{\beta}[U][Z]$, with
deg$_{Z} (Z^d) < d$.
The idea consists
in replacing all powers $Z^{j}, j \geq d$,  
in $\widetilde{f_{\beta}}(Z)$ by 
polynomials in $Z$, of degree $< d$,  
with coefficients
in $\mathbb{K}_{\beta}[U]$.
Let us prove recursively that 
$Z^h 
\in \mathbb{K}_{\beta}[U][Z]$, with
deg$_{Z} (Z^h) < d$,
for all $h \geq d$: 
assume $Z^h := \sum_{i=0}^{d-1} v_{i,h} Z^i$
with $v_{i,h} \in \mathbb{K}_{\beta}[U]$
and show
$Z^{h+1} \in \mathbb{K}_{\beta}[U][Z]$, with
deg$_{Z} (Z^{h+1}) < d$.
Indeed,
$$Z^{h+1} := \sum_{i=0}^{d-1} v_{i,h+1} Z^i
= (Z^h) Z = \sum_{i=0}^{d-2} v_{i,h} Z^{i+1} +
v_{d-1,h} Z^{d}$$
$$
= \sum_{i=0}^{d-2} v_{i,h} Z^{i+1} +
v_{d-1,h} \left[
\frac{1}{\gamma_d} U -
\left(
\frac{\gamma_1}{\gamma_d} Z + \frac{\gamma_2}{\gamma_d} Z^2 + \ldots +
\frac{\gamma_{d-1}}{\gamma_d} Z^{d-1}
\right)
\right].
$$
Hence 
\begin{equation}
\label{coeffCOMPAGNONmatrice}
v_{0,h+1} = \frac{1}{\gamma_d} v_{d-1,h} U \quad{\rm and}\quad
v_{i,h+1} = v_{i-1,h} - \frac{\gamma_i}{\gamma_d} v_{d-1,h}, 
~1 \leq i \leq d-1,
\end{equation}
and the result. We deduce
$$
\widetilde{f_{\beta}}(Z)
=
\sum_{h \geq 1} \lambda_h Z^h
=
\sum_{h = 1}^{d-1} \lambda_h Z^h 
+ \sum_{h \geq d} \lambda_h Z^h 
=
\sum_{i = 1}^{d-1} \lambda_i Z^i
+
\sum_{h \geq d} \lambda_h \left(
\sum_{i=0}^{d-1} v_{i,h} Z^i
\right)
$$
\begin{equation}
\label{coeffGERME}
= \sum_{i=0}^{d-1} \left(
\lambda_i + \sum_{h \geq d} \lambda_h v_{i,h}
\right) Z^i ~\in \mathbb{C}[[U]][Z].
\end{equation}
\end{proof}

Equation \eqref{eqexistGERM} is exactly
\eqref{rewrittingFBETA} with the
usual variable $z$.

We call $G_{\beta}$ the germ associated with
the analytic function
$f_{\beta}(z)$, or with the base of numeration $\beta$.

Following Theorem \ref{decompCONVERGENTthm} 
and the relations
\eqref{coeffCOMPAGNONmatrice} and \eqref{coeffGERME}, 
the decomposition of the germ $G_{\beta}$ shows
that the coefficients of its Puiseux series
do possess a ``right - continuity" property, with $\beta$, via the
functions $\lambda_j$ 
(Proposition \ref{lambda_jCONTINUOUS}), and
an ``asymptotic" property, linked to the invariants
of the companion matrix form of \eqref{coeffCOMPAGNONmatrice}. 
This will be developped further elsewhere.
The interest of such a remark may consist in 
studying globally the properties of the family of
germs $(G_{\beta})$ when $\beta > 1$ varies in the
set of algebraic numbers.

\begin{theorem}
\label{puiseuxFBETA}
Let $\beta > 1$ be an algebraic number,
$P_{\beta}(X)$ its minimal polynomial and 
$G_{\beta}$ the germ associated with the Parry Upper function
$f_{\beta}(z)$. Then
\begin{equation}
\label{deccc}
G_{\beta}(U,Z) ~=~ v \, U \, G_{y_1} G_{y_2} \ldots G_{y_s}
\end{equation}
where $v=v(U,Z) \in \mathbb{C}\{U,Z\}$ is an invertible series,
and the convergent Puiseux series 
$$y_{1}(U) = \sum_{i \geq 1} a_{i,1} U^{i/\nu(y_1)}, \ldots, 
y_{s}(U) = \sum_{i \geq 1} a_{i,s} U^{i/\nu(y_s)}$$ 
are uniquely
determined by $G_{\beta}$, up to conjugation,
with
$$G_{\beta}(P^{*}_{\beta}(z),z-\frac{1}{\beta}) 
~=~
f_{\beta}(z)
~=~
$$
\begin{equation}
\label{decccz}
v(P^{*}_{\beta}(z),z-\frac{1}{\beta}) 
\, P^{*}_{\beta}(z)
\,
\prod_{i=1}^{\nu(y_1)} \bigl(z-\frac{1}{\beta}
- y_{i,1}(P^{*}_{\beta}(z))
\bigr)
\ldots
\prod_{i=1}^{\nu(y_s)} \bigl(z-\frac{1}{\beta}
- y_{i,s}(P^{*}_{\beta}(z))
\bigr),
\end{equation}
and
\begin{equation}
\label{hauteurDEGbeta}
h(\mathcal{N}(G_{\beta})) ~=~ \sum_{i=1}^{s} \nu(y_i)
~<~ {\rm deg} \, \beta.
\end{equation}
\end{theorem}

\begin{proof}
Theorem \ref{decompCONVERGENTthm} is applied to the germ
$G_{\beta}(U,Z)$.
Since $f_{\beta}(z)$ is convergent in a neighbourhood
of $1/\beta$, $G_{\beta}$
and all the Puiseux expansions involved 
in its decomposition are convergent in this
neighbourhood.
The power of $U$ in \eqref{deccc} 
is necessarily equal to $1$ since
$f'_{\beta}(1/\beta) > 0$, i.e. $0$ is a simple zero
of $G_{\beta}(\widetilde{P^{*}_{\beta}}(Z),Z)$.

Since deg$_{Z} G_{\beta}(U,Z) <$ deg $\beta$, by the definition
of the height of the Newton polygon of the germ
$G_{\beta}$, we readily
deduce \eqref{hauteurDEGbeta} 
from Theorem \ref{decompoMAINthm} (ii).
\end{proof}

For $\beta > 1$ any algebraic number,
a beta-conjugate
$\xi$ of $\beta$ is by definition
a complex number such that 
(i) $\xi^{-1}$ is a zero of $f_{\beta}(z)$ 
which lies in its domain of definition, 
(ii) $\xi$ is not a Galois conjugate
of $\beta$.

For Parry numbers $\beta$, 
\eqref{fzetaNONsimple} and \eqref{fzetaSIMPLE} show
that this definition is exactly 
the usual one which uses  
the Parry polynomial of $\beta$ \cite{boyd2}.

Equation \eqref{decccz} gives 
the exhaustive list of zeros of $f_{\beta}(z)$, 
and therefore suggests the following alternate definition of the
beta-conjugates of $\beta$ (where the natural 
boundary $|z|=1$ of
$f_{\beta}(z)$ is taken into account, 
if $\beta$ is not a Parry number).

\begin{definition}
\label{BETACONJUGATEpuiseuxFBETAnewdef}
Let $\beta > 1$ be an algebraic number.

(i) A complex number $\xi$ which satisfied
\begin{equation}
\label{defiBETACONGnew}
0 ~=~ \xi^{-1} - \beta^{-1} - 
\sum_{i \geq 1} a_i \left(P_{\beta}^{*}(\xi^{-1})\right)^{i/n},
\end{equation}
where $y(U) = \sum_{i \geq 1} a_i U^{i/n}$, $n = \nu(y)$, 
is any $Z$-root,
is called a {\it cancellation point} of the germ
$G_{\beta}(U,Z)$. We say that the cancellation point
$\xi$ lies on the $Z$-root $y(U)$.
The set of cancellation points
is denoted by
$\mathcal{S}_{\beta}$. 
Equation \eqref{defiBETACONGnew} has to be understood
as the composition of the (convergent) two analytic functions
$z \to z - \beta^{-1} -
\sum_{i \geq 1} a_i z^i$
and $z \to P_{\beta}^{*}(z)$ with the multivalued 
(polydromic) analytic function
$z \to \Bigl(P_{\beta}^{*}(z)\Bigr)^{1/n}$. Since
$\xi$ is not a Galois conjugate of $\beta$ the function
$z \to P_{\beta}^{*}(z)$ does not cancel 
on a small neighbourhood of
$\xi^{-1}$; this give a sense to \eqref{defiBETACONGnew}.

Since the Puiseux expansions in \eqref{defiBETACONGnew}
are convergent, truncating them to a few terms transforms
\eqref{defiBETACONGnew} into a finite collection of equations
whose solutions provide the geometry of the beta-conjugates
of $\beta$ with a certain approximation, 
controlled by the error terms.
This approach will be continued elsewhere. 

(ii) If $\beta$ is a Parry number, a beta-conjugate
of $\beta$ is a cancellation point of the germ. 
The set $\mathcal{S}_{\beta}$ is the 
set of beta-conjugates 
of $\beta$, and 
$\mathcal{S}_{\beta} \subset \Omega$
Solomyak's fractal.

(iii) If $\beta$ is not a Parry number,
a beta-conjugate
of $\beta$ is a cancellation point
$\xi \in \mathcal{S}_{\beta}$ of the germ
such that $|\xi| > 1$.

(iv) A cancellation point
$\xi \in \mathcal{S}_{\beta}$, lying on
the $Z$-root $y(U)$,
is said Puiseux-conjugated to another 
cancellation point
$\xi' \in \mathcal{S}_{\beta}$
if $\xi'$,  lying on a $Z$-root $y'(U)$, 
is such that $y(U)$ and $y'(U)$ belong to the same
conjugacy class
of the germ $G_{\beta}(U,Z)$.
\end{definition}

If $\beta > 1$ is an algebraic number which is not a Parry number
the natural boundary $|z|=1$ 
of $f_{\beta}(z)$ is the natural boundary
of at least one of the factors in \eqref{decccz}, 
but not necessarily of all of them a priori. 
In other terms it may occur that
Puiseux-conjugation may be addressed to 
cancellation points of the germ
$G_{\beta}$ which lie beyond the natural boundary of  
$f_{\beta}(z)$, some branches possibly 
spiraling outside the domain of definition
of $f_{\beta}(z)$.

\section{Rationality, descent over $\mathbb{Q}$, and 
factorization of the Parry polynomial of a Parry number}
\label{S6}

Let $\beta$ be a Parry number, with $m$ 
as preperiod lenght and $p+1$ as period length in
$d_{\beta}(1)$. Then
the Parry polynomial of $\beta$ is,
for non-simple Parry numbers,
$$P_{\beta,P}(X) =  
X^{m+p+1} -t_1 X^{m+p} - t_2 X^{m+p-1} - \ldots - t_{m+p} X - t_{m+p+1}
\hspace{2cm} \mbox{}$$
\begin{equation}
\label{nonsimplepoly1}
\mbox{} \hspace{3cm}- X^{m} +t_1 X^{m-1} + t_2 X^{m-2} + \ldots + t_{m-1} X +t_m
\end{equation}
and
\begin{equation}
\label{nonsimplepoly2}
P_{\beta,P}(X) = X^{p+1} -t_1 X^{p} - t_2 X^{p-1} - \ldots - t_{p} X - (1+t_{p+1})
\end{equation}
in the case of pure periodicity. 
For simple Parry numbers, the Parry polynomial
is
\begin{equation}
\label{simplepoly3}
P_{\beta,P}(X) = X^{m} -t_1 X^{m-1} - t_2 X^{m-2} - \ldots - t_{m-1} X - t_{m}
\end{equation}
with $m \geq 1$ \cite{frougny1} \cite{lothaire} \cite{vergergaugry2}.
The height ($=$ maximum of the moduli of the coefficients) 
of the Parry polynomial lies in
$\{\lfloor \beta \rfloor ,
\lceil \beta \rceil\}$;
if $\beta$ is a simple Parry number, then it is equal to
$\lfloor \beta \rfloor$ \cite{vergergaugry2}.
In the decomposition of $P_{\beta,P}(X)$ 
as the product
of irreducible polynomials with coefficients in
$\mathbb{Q}$, as
$$P_{\beta,P} = P_{\beta} \, \pi_1 \, \pi_2 \ldots \,\pi_{\sigma} ,$$
we may identify the irreducible factors ~$\pi_j$
as arising from the 
conjugacy classes of the germ 
$G_{\beta}$. This requires some assumptions.

\begin{theorem}
\label{factoPARRYpoly}
Let $\beta > 1$ be a Parry number,
$P_{\beta}(X)$ its minimal polynomial, 
$P_{\beta,P}(X)$ its Parry polynomial decomposed
as $P_{\beta,P} = P_{\beta} \, \pi_1 \, \pi_2 \ldots \,\pi_{\sigma} $
into irreducible factors. Let 
$G_{\beta}$ the germ associated with 
$\beta$ and $L$ be the field of coefficients
of the Puiseux series of $G_{\beta}$. 
Assume that all the Puiseux expansions of $X$
in $G_{\beta}$ are distinct.
Assume 
$[L : \mathbb{Q}] < +\infty$ and, for each 
$L$-rational conjugacy class $C$, the
product
$$\prod_{y_i \in C} (Y - y_{i}(X))
\qquad \mbox{lies in~} \, \mathbb{Q}[X][Y].$$ 
If $e$ is the number of
$L$-rational conjugacy classes
$(C_j)_{j=1,\ldots,e}$, then 

(i) $e = \sigma < {\rm deg} \, \beta$,
and 

(ii) up to the order,  
\begin{equation}
\label{deccc}
\pi_{j}^{*}(X) ~=~
\prod_{y_i \in C_j} (X - \frac{1}{\beta} - y_{i}(P_{\beta}^{*}(X))) 
\, , \quad j = 1, \ldots, e.
\end{equation}
\end{theorem}

\begin{proof}
This is a consequence of Proposition 2.1 in Walsh \cite{walsh}.
Under the present assumptions $\pi_i \neq \pi_j$
if $i \neq j$ and the decomposition of
$G_{\beta}$, as given by \eqref{GgermCONVERGENT},  
allows to write $f_{\beta}(z)$ as a product of 
distinct irreducible
factors in $\mathbb{Q}[X][Y]$. From
\eqref{fzetaNONsimple} and
\eqref{fzetaSIMPLE}
the identification of the
factors readily gives 
$e = \sigma$, the irreducible factors 
$\pi_{j}^{*}$ and the unit
$v = -(1 - z^{k})^{-1}$, with $k = m$
if $\beta$ is simple, with
$d_{\beta}(1)$ of length $m$, and
$k=p+1$ if $\beta$ is not simple, with
$d_{\beta}(1)$ of period length $p+1$.

From \eqref{hauteurDEGbeta},
the number $\sigma$ of irreducible factors which arises from 
$L$-rational conjugacy classes of 
Puiseux expansions is smaller than 
deg $\beta$.
\end{proof}

\section{A product formula for $\zeta_{\beta}(z), \beta$ a Parry number}
\label{S7}

Using \eqref{fzetaNONsimple} and 
\eqref{fzetaSIMPLE} 
and assuming the hypotheses of
Theorem \ref{factoPARRYpoly} we obtain the following reformulation
of the dynamical zeta function $\zeta_{\beta}(z)$ 
as a finite product over
the $e$ $L$-rational conjugacy classes, $e < $ deg $\beta$,
\begin{equation}
\label{zetaPRODUCT++}
\zeta_{\beta}(z) ~=~ 
v \, \frac{1}{P_{\beta}^{*}(z)} \,
\prod_{j=1}^{e}
\left(
\frac{1}{\prod_{y_i \in C_j} \Bigl(
z-\frac{1}{\beta} - y_{i}(P_{\beta}^{*}(z))
\Bigr)}
\right). 
\end{equation}
The unit
$v$ is equal to $(1-z^k)$ with
$k = m$ if $\beta$ is simple, with
$d_{\beta}(1)$ of length $m$, and
$k=p+1$ if $\beta$ is not simple, with
$d_{\beta}(1)$ of period length $p+1$. 
The poles of $\zeta_{\beta}(z)$ are 
either the reciprocals $\xi^{-1}$
of the cancellation points 
$\xi$ of the germ $G_{\beta}$
of $\beta$, or the reciprocals of the Galois conjugates
of $\beta$. 

The assumptions in 
Theorem \ref{factoPARRYpoly} could 
probably be weakened, for obtaining
the same decomposition \eqref{zetaPRODUCT++}.

\frenchspacing

\section*{Acknowledgements}

The author is indebted to M. Pollicott 
and to M. Lejeune-Jalabert
for valuable comments and discussions.

\vspace{2cm}

Jean-Louis Verger-Gaugry,

Institut Fourier, CNRS UMR 5582, 

Universit\'e Jospeh Fourier Grenoble I, 

BP 74, 38402 Saint-Martin d'H\`eres, France.

email:\,\tt{jlverger@ujf-grenoble.fr}

\end{document}